\documentclass{article}

\RequirePackage[OT1]{fontenc}
\RequirePackage[%lnms,
amsthm,amsmath%,natbib
]{imsart}

\usepackage{graphicx}
\usepackage{latexsym,amsmath}
\usepackage{amsmath,amsthm,amscd}
\usepackage{amsfonts}
\usepackage[psamsfonts]{amssymb}
\usepackage{enumerate}
\usepackage[usenames]{color}

\usepackage{url}

\usepackage{tocvsec2}

\usepackage{epstopdf}
\usepackage{wrapfig}
\usepackage{float}

\startlocaldefs
\numberwithin{equation}{section}

\theoremstyle{plain} 
\newtheorem{theorem}{Theorem}[section]
\newtheorem{corollary}[theorem]{Corollary}

\newtheorem{proposition}[theorem]{Proposition}

\theoremstyle{definition} 

\theoremstyle{definition} 

\newtheorem*{ex*}{Example}
\theoremstyle{remark} 

\theoremstyle{remark} 

\newtheorem*{remark*}{Remark}
\numberwithin{equation}{section}
 
\newcommand{\beqa}{\begin{eqnarray}}
\newcommand{\eeqa}{\end{eqnarray}}
 
\newcommand{\bseq}{\begin{subequations}}
 
\newcommand{\eseq}{\end{subequations}}

\newcommand{\dd}{\partial}

\renewcommand{\dd}{{\,\operatorname{d}}}

\newcommand{\al}{\alpha}

\newcommand{\Ga}{\Gamma}
\newcommand{\si}{\sigma}

\newcommand{\de}{\delta}
\newcommand{\be}{\beta}

\renewcommand{\Psi}{\overline{\Phi}}

\newcommand{\up}{\nearrow}

\renewcommand{\up}{\operatorname{\mathsf{up}}}
\newcommand{\lo}{\operatorname{\mathsf{lo}}}

\newcommand{\ii}[1]{\,\mathbf{I}\{#1\}}

\newcommand{\fd}[2]{\frac{\dd#1}{\dd#2}}

\newcommand{\U}{\mathcal{U}}
\renewcommand{\L}{\mathcal{L}}

\newcommand{\vp}{\varepsilon}

\newcommand{\tell}{{\tilde{\ell}}}

\newcommand{\tu}{{\tilde{u}}}

\renewcommand{\le}{\leqslant}
\renewcommand{\ge}{\geqslant}

\endlocaldefs

\begin{document}

\begin{frontmatter}

\title{Geometrically convergent sequences of upper and lower bounds on the Wallis ratio and related expressions}
\runtitle{Bounds on the Wallis ratio}
%\date{\today}

% \author{\fnms{First}  \snm{Author}\corref{}\thanksref{t2}\ead[label=e1]{first@somewhere.com}},
%  \author{\fnms{Second} \snm{Author}\ead[label=e2]{second@somewhere.com}}
%  \and
%  \author{\fnms{Third}  \snm{Author}%
%  \ead[label=e3]{third@somewhere.com}%
%  \ead[label=u1,url]{http://www.foo.com}}
%
%  \thankstext{t2}{Footnote to the first author with the `thankstext' command.}

\begin{aug}
\author{\fnms{Iosif} \snm{Pinelis}\thanksref{t2}\ead[label=e1]{ipinelis@mtu.edu}}
  \thankstext{t2}{Supported by NSF grant DMS-0805946}
\runauthor{Iosif Pinelis}

%\affiliation{Michigan Technological University}

\address{Department of Mathematical Sciences\\
Michigan Technological University\\
Houghton, Michigan 49931, USA\\
E-mail: \printead[ipinelis@mtu.edu]{e1}}
\end{aug}

\begin{abstract}
Sequences of algebraic upper and lower bounds on the Wallis ratio $\Ga(x+1)/\Ga(x+\frac12)$ are given with the relative errors that converge to $0$ geometrically and uniformly on any interval of the form $[x_0,\infty)$ for $x_0>-\frac12$; moreover, the relative and absolute errors converge to $0$ as $x\to\infty$. 
These conclusions are based on corresponding results for the digamma function $\psi:=\Ga'/\Ga$. 
Relations with other relevant results are discussed, as well as the corresponding computational aspects. 
This work was motivated by studies of exact bounds involving the Student probability distribution. 
\end{abstract}

%\subjclass[2000]{60E15, 62G10, 62G15, 60G50, 62G35}
% 62G10    	Hypothesis testing
%  62G15    	Tolerance and confidence regions
%  60G50    	Sums of independent random variables; random walks
%   62G35    	Robustness
  
%
%\keywords{probability inequalities; Rade\-macher random variables; sums of independent random variables; Student's test; self-normalized sums}

%33B15   	Gamma, beta and polygamma functions
%26D07   	Inequalities involving other types of functions
%26D15   	Inequalities for sums, series and integrals
%41A17   	Inequalities in approximation (Bernstein, Jackson, Nikol?skii(-type inequalities)
%33F05   	Numerical approximation and evaluation [See also 65D20]
%65D20   	Computation of special functions, construction of tables [See also 33F05]

%62E15   	Exact distribution theory
%62E17   	Approximations to distributions (nonasymptotic)
%62E20   	Asymptotic distribution theory
%60E15   	Inequalities; stochastic orderings


% 60B11    	Probability theory on linear topological spaces
% 46B09    	Probabilistic methods in Banach space theory [See also 60Bxx]
% 46B20    	Geometry and structure of normed linear spaces
% 46B10    	Duality and reflexivity [See also 46A25]


\end{frontmatter}

\settocdepth{chapter}

\tableofcontents 
%%%%%%%%%%%%%%%%%{\small\tableofcontents} 

\settocdepth{subsubsection}

\theoremstyle{plain} 
\numberwithin{equation}{section}

\eject

\section{Summary and discussion}\label{intro} 

%\subsection{Summary}\label{summary} 
My interest to the Wallis ratio 
%The Wallis ratio may be defined as % by the formula 
\begin{equation*}
	W(x):=\frac{\Ga(x+1)}{\Ga(x+\frac12)}  
\end{equation*}
was stimulated by recent work \cite{BE-student,closeness-student} on exact bounds involving Student's probability distribution, with the density function $f_p$ defined by the formula
\begin{align*}
	f_p(x)&:=\frac{ \Gamma
   \left(\frac{p+1}{2}\right)}{\sqrt{\pi p}\, \Gamma \left(\frac{p}{2}\right)}\,
   \left(1+\frac{x^2}{p}\right)^{-(p+1)/2} 
\end{align*}
for all real $x$; here $p$ is a positive parameter referred to as the number of degrees of freedom. 
Let us extend this definition by continuity to $p=\infty$, so that $f_\infty$ is the probability density function of the standard normal distribution: 
$%\begin{equation}
	f_\infty(x)=\frac1{\sqrt{2\pi}}\,e^{-x^2/2}
$ %\end{equation}
for all real $x$. 
Thus, 
\begin{equation}\label{eq:r=W}
r(p):=\frac{f_p(0)}{f_\infty(0)}=\sqrt{\frac2p}\,W\Big(\frac{p-1}2\Big)	
\end{equation}
for all $p\in(0,\infty)$; equivalently, 
\begin{equation}\label{eq:W=r}
	W(x)=r(2x+1)\sqrt{x+\frac12}
\end{equation}
for all $x\in(-\frac12,\infty)$. 
Therefore, bounding the Wallis ratio $W(x)$ is equivalent to bounding the ratio $r(p)$, which is what we shall do, as the bounds on $r(p)$ are slightly simpler to express and easier to operate with. 
The ratio $r(p)$ may also be slightly more natural, since $r(p)\to1$ as $p\to\infty$.  

\begin{theorem}\label{th:}
Take any $p\in(0,\infty)$. Then one has the identities 
\begin{equation}\label{eq:identity}
	r(p)=U_\infty(p)=L_\infty(p),    
\end{equation} 
where 
\begin{align}%
	U_k(p)&:=\exp\sum_{m=1}^k 2^{-1-m} \si_{p,m}, \label{eq:U} \\ 
	\si_{p,m}&:=\sum _{j=0}^m (-1)^j \binom{m}{j}\,\ln(p+j), \label{eq:si} \\ %\quad\text{and}\quad
	  L_k(p)&:=\sqrt{\frac p{p+1}}\,\frac1{U_k(p+1)}, \label{eq:LL} %\\ %\quad\text{and}\quad 
\end{align} 
and $k\in\{0,1,\dots,\infty\}$; 
the series $\sum_{m=1}^\infty 2^{-1-m} \si_{p,m}$ converges uniformly in $p\in[p_0,\infty)$, for any $p_0\in(0,\infty)$. 
\big(As usual, the sum of an empty family is defined to be $0$; so, $U_0(p)=1$.\big) 

Moreover, for all $k\in\{0,1,\dots\}$ the bracketing inequalities 
\begin{equation}\label{eq:L<r<U}
	L_k(p)<L_{k+1}(p)<r(p)<U_{k+1}(p)<U_k(p) 
\end{equation}
hold. For the relative errors 
\begin{align*}
	\de_{\up,k}(p)&:=\frac{U_k(p)-r(p)}{r(p)}=\Big|\frac{U_k(p)-r(p)}{r(p)}\Big|\quad\text{and}
	%\label{eq:de_up}
	\\ 
	\de_{\lo,k}(p)&:=\frac{r(p)-L_k(p)}{r(p)}=\Big|\frac{L_k(p)-r(p)}{r(p)}\Big| %\label{eq:de_lo}
\end{align*}
of the upper and lower approximations one has 
\begin{equation}\label{eq:err bound}
	0<\de_{\lo,k}(p)<\de_{\up,k}(p)<\exp(\rho^*_{p,k})-1, 
\end{equation}
where %$k>2-2p$ and 
\begin{equation}\label{eq:rrho}
	\rho^*_{p,k}:=\frac1{2^{k+1}}\,\Big(\frac{(k+1)!}{p^{k+1}}\bigwedge\ln\frac{p+1}p\Big), 
\end{equation}
which converges to $0$ as $k\to\infty$ (no slower than geometrically) uniformly in $p\in[p_0,\infty)$, for any $p_0\in(0,\infty)$, and also converges, for each $k\in\{0,1,\dots\}$, to $0$ as $p\to\infty$. 
\end{theorem}

The necessary proofs are deferred to Section~\ref{proofs}. 

Note that the definition \eqref{eq:LL} of the lower bound $L_k(p)$ originates in the identity
\begin{equation}\label{eq:r(p)r(p+1)}
r(p)r(p+1)=\sqrt{\frac p{p+1}} 	
\end{equation}
for all $p>0$, 
which easily follows from $\Ga(x+1)=x\Ga(x)$. 
Equivalently, 
\begin{equation}\label{eq:W(x)W(x+1/2)}
W(x)W(x+\tfrac12)=x+\tfrac12 	
\end{equation}
for all $x>-\frac12$. 
Thus, given an upper bound on $W(x)$ \big(or on $r(p)$\big), one automatically has a corresponding lower bound --- and vice versa; this is a well-known trick. 

An important feature of the lower and upper bounds $L_k(p)$ and $U_k(p)$ is that they are all algebraic. Therefore, as is done in \cite{closeness-student}, one can use the well-known result by Tarski \cite{tarski48,collins98} to solve, in a completely algorithmic manner, systems of equalities/inequalities involving such bounds. 
Let us list here the initial 5 upper bounds on $r(p)$:
\begin{gather*}
U_0(p)=1,
\quad
U_1(p)=\Big(\frac1{1+a}\Big)^{1/4}, \quad
U_2(p)=\Big(\frac{1+2a}{(1+a)^4}\Big)^{1/8}, \\
U_3(p)=\Big(\frac{(1+2a)^5}{(1+a)^{11}(1+3a)}\Big)^{1/16}, \quad
U_4(p)=\Big(\frac{(1+2a)^{16} (1+4a)}{(1+a)^{26} (1+3a)^6}\Big)^{1/32}, 
\end{gather*}
where $a$ stands for $1/p$. 
Recall that the corresponding lower bounds can be obtained by \eqref{eq:LL}. 

Recalling also \eqref{eq:W=r} and assuming the correspondence 
\begin{equation}\label{eq:corr}
	(-\tfrac12,\infty)\ni\tfrac{p-1}2=x\longleftrightarrow p=2x+1\in(0,\infty), 
\end{equation}
introduce 
\begin{equation}\label{eq:UU,LL}
		\U_k(x)=U_k(p)\sqrt{p/2} \quad\text{and}\quad 
	\L_k(x)=L_k(p)\sqrt{p/2}. 
\end{equation}
Then 
one immediately has the following corollary of Theorem~\ref{th:}: 

\begin{corollary}\label{cor:W}
Relations \eqref{eq:identity} and \eqref{eq:L<r<U} can be rewritten as 
\begin{equation}\label{eq:W-identity&ineqs}
	\L_k(x)<\L_{k+1}(x)<\L_\infty(x)=W(x)=\U_\infty(x)<\U_{k+1}(x)<\U_k(x)    
\end{equation} 
for all $x\in(-\tfrac12,\infty)$ and $k\in\{0,1,\dots\}$. 
Moreover, the relative errors of the bounds $\U_k(x)$ and $\L_k(x)$ 
converge geometrically and uniformly on any interval of the form $[x_0,\infty)$ for any $x_0>-\frac12$; also, the relative errors converge to $0$ as $x\to\infty$.  
Furthermore, it follows from \eqref{eq:err bound}--\eqref{eq:rrho} and the inequality $W(x)<\U_0(x)=\sqrt{x+\frac12}$ that, for each $k\in\{0,1,\dots\}$,   
the absolute errors $\U_k(x)-W(x)$ and $W(x)-\L_k(x)$ 
converge to $0$, too, as $x\to\infty$. 
\end{corollary}

Theorem~\ref{th:} is based on the following proposition, which provides geometrically convergent sequences of rational upper and lower bounds on $\psi(x+1)-\psi(x+\tfrac12)$ and thus 
may be of independent interest. 
Here, as usual, $\psi$ denotes the digamma function, that is, the logarithmic derivative of $\Ga$: 
\begin{equation}\label{eq:psi}
	\psi:=\frac{\Ga'}{\Ga}.  
\end{equation}

\begin{proposition}\label{prop:r'}
Take any $p\in(0,\infty)$ and $k\in\{0,1,\dots\}$. 
Then 
\begin{equation}\label{eq:lnr'-identity&ineqs} % \ell:=(old \tJ/p)
	\ell_k(p)<\ell_{k+1}(p)<\ell_\infty(p)=\fd{\ln r(p)}p=u_\infty(p)<u_{k+1}(p)<u_k(p),    
\end{equation} 
%for all $p\in(0,\infty)$ and $k\in\{0,1,\dots,\infty\}$, 
where, for $k\in\{0,1,\dots,\infty\}$,  
\begin{equation}\label{eq:ell,u}
	\ell_k(p):=\sum _{m=1}^k \frac{2^{-1-m} m!}{p^{(m+1)}}, \quad %\text{and}\quad
	u_k(p):=\frac1{2p(p+1)}-\ell_k(p+1), 
\end{equation}
and  
$$y^{(m)}:=\prod _{j=0}^{m-1}(y+j)=\frac{\Ga(y+m)}{\Ga(y)}$$ 
is the Pochhammer symbol, with $y^{(0)}=1$, corresponding to the general convention that the product of an empty family is defined as $1$. 
Equivalently,
\begin{equation}\label{eq:psi-identity&ineqs} % \ell:=(old \tJ/p)
	\tell_k(x)<\tell_{k+1}(x)<\tell_\infty(x)=\psi(x+1)-\psi(x+\tfrac12)=\tu_\infty(x)<\tu_{k+1}(x)<\tu_k(x)  
\end{equation}  
for all $x\in(-\tfrac12,\infty)$ and $k\in\{0,1,\dots\}$, where, for all $k\in\{0,1,\dots,\infty\}$,   
\begin{equation*}
	\tell_k(x):=\tfrac1p+2\ell_k(p),\quad\text{and}\quad
	\tu_k(x):=\tfrac1p+2u_k(p),  
\end{equation*} 
again assuming the correspondence \eqref{eq:corr}. 
Moreover, one has the following bounds on the errors: %if $k>2-2p$ and $p>0$ then 
\begin{align}
	0<%e_k(p):=
	u_k(p)-\fd{\ln r(p)}p<\fd{\ln r(p)}p-\ell_k(p)
	&<\frac{2^{-k-1}(k+1)!}{p^{(k+2)}} %\label{eq:psi-error1} 
	&<\frac{2^{-k-1}}p, \label{eq:psi-error}
\end{align}
which converges to $0$ as $k\to\infty$ (no slower than geometrically) uniformly in $p\in[p_0,\infty)$, for any $p_0\in(0,\infty)$, and also converges, for each $k\in\{0,1,\dots\}$, to $0$ as $p\to\infty$.  
\end{proposition}

Let us now go back to the bounds $U_k(p)$ and $L_k(p)$, defined in \eqref{eq:U} and \eqref{eq:LL}. 
It is not hard to see that they follow some simple recursive relations: 
%, by which they can be easily computed: 

%\newpage

\begin{proposition}\label{prop:recur}
Take any $p>0$ and $k\in\{1,2,\dots\}$. Then 
\begin{alignat}{2}
	U_k(p)=&\sqrt{U_{k-1}(p)L_{k-1}(p)},&\quad
	L_k(p)=&\sqrt{\frac{\frac p{p+1}}{U_{k-1}(p+1)L_{k-1}(p+1)}}, \label{eq:U,L} \\  
		V_k(p)=&\sqrt{p\,\frac{V_{k-1}(p)}{V_{k-1}(p+1)} },&\quad
		M_k(p)=&\sqrt{p\,\frac{M_{k-1}(p)}{M_{k-1}(p+1)} }, \label{eq:V,M}
\end{alignat}
where  
\begin{equation}\label{eq:V,M defs}
	V_k(p):=\sqrt{p}\,U_k(p)\quad\text{and}\quad M_k(p):=\sqrt{p}\,L_k(p).  
\end{equation} 
The ``initial'' conditions are $U_0(p)=1$, $L_0(p)=\sqrt{\frac p{p+1}}$, $V_0(p)=\sqrt p$, and $M_0(p)=\frac p{\sqrt{p+1}}$. 
\end{proposition}

One possible way to apply Proposition~\ref{prop:recur} is to use the first recursive relation in \eqref{eq:V,M} to compute $V_k(p)$, and then get $U_k(p)$ and $L_k(p)$ by \eqref{eq:V,M defs} and \eqref{eq:LL}.   

%In view of \eqref{eq:LL}, a similar recursive relation exists between the lower bounds $L_k(p)$. 
%However, it may usually be more efficient to compute $U_k(p)$ first and then use \eqref{eq:LL} to find $L_k(p)$. 

In view of \eqref{eq:U} and \eqref{eq:si}, one can also rewrite $U_k(p)$ as %notes of 1/10/11 
\begin{equation}\label{eq:U,H}
	U_k(p)=\Big(\prod_{j=0}^k (p+j)^{(-1)^j H_{j,k}}\Big)^{1/2^{k+1}}, 
\end{equation}
where $H_{j,k}:=\sum_{m=1\vee j}^k 2^{k-m}\binom mj$ can be found recursively: 
\begin{equation*}%\label{eq:H recur}
	H_{j,j}=\ii{j\ge1}\quad\text{and}\quad
	H_{j,k+1}=2H_{j,k}+\binom{k+1}j
\end{equation*}
for nonnegative integers $j$ and $k$ such that $j\le k$; 
here, as usual, $\ii{\mathcal A}$ denotes the indicator of an assertion $\mathcal A$, so that $\ii{\mathcal A}$ equals $1$ if $\mathcal A$ is true and equals $0$ otherwise. 
Recalling also the known recursive relation \eqref{eq:binom} for the binomial coefficients, one sees that the calculation of the bound $U_k(p)$ takes $O(k^2)$ operations such as addition and multiplication, followed by raising the result to the power $1/2^{k+1}$; here we used the fact that is takes $O(\log N)$ multiplications to raise a given real number to the power of a natural number $N$.  
Since the relative errors $\de_{\lo,k}(p)$ and $\de_{\up,k}(p)$ are $O(1/2^k)$, it takes $O(\log^2\frac1\vp)$ operations to make the relative errors less than $\vp\in(0,1)$. 
 
Moreover, observe that the integers $H_{j,k}$ do not depend on $p$. 
Therefore, in rather typical situations (such as plotting or numerically integrating) when one would need to compute the bounds $U_k(p)$ and $L_k(p)$ for the same $k$ but many different values of $p$, the $H_{j,k}$'s need to be computed just once, and then they can be saved and used for all the different values of $p$. 
So, using \eqref{eq:U,H} to compute $U_k(p)$ for each of the many values of $p$, one will need to perform just $O(k)$ operations --- raising $p,\dots,p+k$ to the powers $H_{0,k},\dots,H_{k,k}$ or, what is essentially the same, compute $O(k)$ logarithms, namely, $\ln p,\dots,\ln(p+k)$ --- followed by $O(k)$ multiplications and additions. 
In some computational environments (such as Mathematica), the calculation of the logarithm of a given number may be about as fast as that the addition of two real numbers. 
Thus, one will need just $O(\log\frac1\vp)$ \big(rather than $O(\log^2\frac1\vp)$\big) operations to make the relative errors $\de_{\lo,k}(p)$ and $\de_{\up,k}(p)$ less than $\vp\in(0,1)$. 

There are a very large number of upper and lower bounds on the Wallis ratio and on the more general ratios of the form $\Ga(x+s)/\Ga(x)$, as well as related bounds for the digamma function $\psi$, as in \eqref{eq:psi}. The survey by Qi \cite{qi} lists 204 sources (only a rather small minority of which concern just background information). See e.g.\ \cite{boyd,chu62,gurland,kazarinoff,kershaw,mortici,slavic,watson}. 
However, there appear to be very few known convergent sequences of upper and lower bounds such as those presented in Theorem~\ref{th:} and Proposition~\ref{prop:r'} here. 

A convergent sequence of lower bounds on the Wallis ratio (and, hence, by the identity \eqref{eq:W(x)W(x+1/2)}, a convergent sequence of upper bounds) can be obtained immediately from the identity 
$W(x)=\sqrt{{}_2F_1(-\frac12,-\frac12;x;1)\,x}$ due to Watson \cite{watson}, which is in turn a corollary of a theorem by Gauss \cite[Theorem~2.2.2]{andrews} for the hypergeometric function. 
However, the hypergeometric series converges rather slowly, so that the number of operations it takes to make the relative errors less than $\vp\in(0,1)$ using the Gauss--Watson identity
is 
on the order of \rule[-6pt]{0pt}{10pt} $\big(\frac1\vp\big)^{1/(x+2)}$, which, for any fixed $x>-\frac12$ and $\vp\downarrow0$, is asymptotically much greater than the mentioned numbers $\log^2\frac1\vp$ and $\log\frac1\vp$. 
Also, with our approach one has the extra bonus of the upper and lower bounds on the difference \break 
$\psi(x+1)-\psi(x+\tfrac12)$, as provided by Proposition~\ref{prop:r'}. 

One should also mention the sequences of lower and upper bounds 
$\al_k(x,s):=(x+s)^{1-s}(x+s+k-1)^{(k)}/(x+k)^{(k)}$ and $\be_k(x,s):=(x+k+s)^{1-s}(x+s+k-1)^{(k)}/(x+k)^{(k)}$ on the more general ratio $\frac{\Ga(x+1)}{\Ga(x+s)}$ for any $s\in(0,1)$ 
given in \cite{shanbhag} (the same bounds, in different notation, were given also in \cite{merkle99}); however, 
$\frac{\be_k(x,s)}{\al_k(x,s)}-1
=\big(\frac{x+k+s}{x+k}\big)^{1-s}-1\sim\frac{s(1-s)}{x+k}$ as $k\to\infty$, 
so that at least one of the two relative errors when using this result is on the order of $\frac1k$ and hence the required number of operations is very large, on the order of $\frac1\vp$. 

In \cite{slavic,merkle96}, asymptotic expansions for $W(x)$ are given, with the relative errors diverging in $k$, but converging to $0$ as $x\to\infty$. 
Also, the sequences of bounds given in \cite{rao_upp} are shown in \cite{shanbhag} to diverge monotonically in $k$ away from the target.  

Sequences of upper and lower bounds on $\psi(x)$ were given in \cite[(36)]{merkle99}, where the difference between the $k$th upper bound and the $k$th lower bound is on the order of $\frac1{(x+k)^2}$, which is much greater than the geometrically decreasing absolute errors of the upper and lower bounds $\tu_k(x)$ and $\tell_k(x)$ on $\psi(x+1)-\psi(x+\tfrac12)$ in \eqref{eq:psi-identity&ineqs}. 

Using methods similar to the ones presented here, one can obtain sequences of upper and lower bounds on the more general expressions of the forms $\Ga(x+s)/\Ga(x)$ and $\psi(x+s)-\psi(x)$ for $s\in(0,1)$, which converge faster to the target than the ones in the existing literature briefly reviewed here, but not as fast as geometrically. These results will be given elsewhere. 

Graphs of the (signed) relative errors $\frac{\U_k(x)-W(x)}{W(x)}$ and $\frac{\L_k(x)-W(x)}{W(x)}$ \big(cf.\ \eqref{eq:UU,LL} and \eqref{eq:W-identity&ineqs}\big) for $k\in\{1,2,3\}$ and $x\in(-\frac12,16]$ are shown as the matrix of plots in Figure~\ref{fig:}; 
the columns of this matrix correspond to the values of $k$ (shown in the plot labels), while the rows correspond to the different intervals of values of $x$: $(-\frac12,2]$, $(2,6]$, and $(6,16]$. 

\begin{figure}[H] %[htbp]
	\centering
		\includegraphics[width=1.00\textwidth]{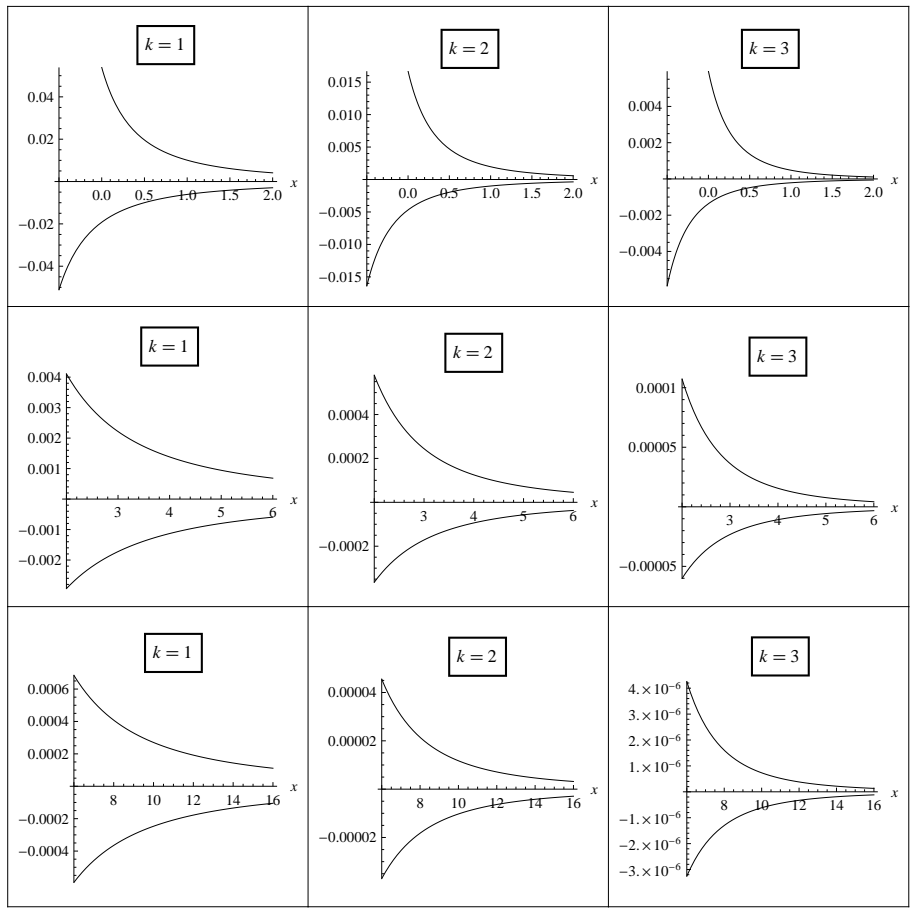}
	\caption{The relative errors of the bounds $\L_k(x)$ and $\U_k(x)$ for $k\in\{1,2,3\}$ and $x\in(-\frac12,16]$.}
	\label{fig:}
\end{figure}

It is seen that the relative errors decrease fast in $k$ and in $x$; moreover, they are rather decent even for $k=1$ and $x\in(-\frac12,2]$. 
Note also that, in view of the identity 
\begin{equation*}
	W(x)=W(x+m)\frac{(x+\frac12)^{(m)}}{(x+1)^{(m)}}
\end{equation*}
for all natural $m$ and all $x\in(-\frac12,\infty)$, without loss of generality one may assume $x$ to be large.

\section{Proofs}\label{proofs} 

Here we shall prove Proposition~\ref{prop:recur}, Proposition~\ref{prop:r'}, and Theorem~\ref{th:} -- in this order. 

\begin{proof}[Proof of Proposition~\ref{prop:recur}] 
Using the identity 
\begin{equation}\label{eq:binom}
\binom{m}{j}=\binom{m-1}{j-1}+\binom{m-1}{j}	
\end{equation}
(with the usual conventions that $\binom{m-1}{-1}=\binom{m-1}{m}=0$), one can see that 
$\si_{p,m}=\si_{p,m-1}-\si_{p+1,m-1}$ for all natural $m$, 
where $\si_{p,m}$ is the same as in \eqref{eq:si}, with $\si_{p,0}=\ln p$. 
In turn, this leads to the identity 
\begin{equation}\label{eq:S}
	S_{p,k}=\tfrac14\,\ln\tfrac p{p+1}+\tfrac12\,(S_{p,k-1}-S_{p+1,k-1})
\end{equation}
for all natural $k$, where $S_{p,k}:=\ln U_k(p)=\sum_{m=1}^k 2^{-1-m} \si_{p,m}$, with $S_{p,0}=0$. 
The identity \eqref{eq:S} can be rewritten as 
\begin{equation}\label{eq:U from S}
U_k(p)=\sqrt{\frac{\sqrt{p}\,U_{k-1}(p)}{\sqrt{p+1}\,U_{k-1}(p+1)} },  	
\end{equation}
which is obviously equivalent to the first identity in \eqref{eq:V,M}. 
Also, in view of \eqref{eq:LL}, \eqref{eq:U from S} can be rewritten as the first identity in \eqref{eq:U,L}. 
Using the latter identity with $p+1$ instead of $p$ to replace $U_k(p+1)$ in 
\eqref{eq:LL} by $\sqrt{U_{k-1}(p+1)L_{k-1}(p+1)}$, one obtains the second identity in \eqref{eq:U,L}. 
In the latter identity, use again \eqref{eq:LL} to replace $U_{k-1}(p+1)$ by 
$\sqrt{\frac p{p+1}}\big/L_{k-1}(p)$; this, together with the definition of $M_k(p)$ in \eqref{eq:V,M}, leads to the second identity in \eqref{eq:V,M}. 
To check the ``initial'' conditions is straightforward. 
Thus, Proposition~\ref{prop:recur} is proved. 
\end{proof}

\begin{proof}[Proof of Proposition~\ref{prop:r'}] 
The four inequalities in \eqref{eq:lnr'-identity&ineqs} trivially follow from the definitions \eqref{eq:ell,u}. 
Relations \eqref{eq:psi-identity&ineqs} follow by \eqref{eq:lnr'-identity&ineqs}, \eqref{eq:r=W}, and \eqref{eq:corr}, which imply $\psi(x+1)-\psi(x+\tfrac12)=\frac1p+2\fd{\ln r(p)}p$. 
Thus, to complete the proof of Proposition~\ref{prop:r'}, it suffices to verify the inequalities \eqref{eq:psi-error} \big(which will, in particular, imply the equalities in \eqref{eq:lnr'-identity&ineqs}\big). 

%Inequality \eqref{eq:psi-error2} is obvious. 
%Take indeed any $p\in(0,\infty)$ and $k\in\{0,1,\dots,\infty\}$. 
Using the Dirichlet formula \cite[Theorem~1.6.1]{andrews} and then the substitution $v:=\frac1{1+z}$, one has 
\begin{equation}
	\psi(x+\de)-\psi(x)=\int_0^\infty\Big(\frac1{(1+z)^x}-\frac1{(1+z)^{x+\de}}\Big)\frac{\dd z}z
	=\int_0^1 v^{x-1}\frac{1-v^\de}{1-v}\dd v
\end{equation}
for any positive $x$ and $\de$. 
Now the substitution $t:=v^{1/2}$ followed by an integration by parts yields 
\begin{equation}\label{eq:lnr'=J}
	2p\fd{\ln r(p)}p
	=p\Big[\psi\Big(\frac{p+1}2\Big)-\psi\Big(\frac p2\Big)\Big]-1
	=2\int_0^1 \frac{pt^{p-1}\dd t}{1+t}-1
	=2J_{p,0}, 
\end{equation}
where
\begin{equation}\label{eq:J}
	J_{p,k}:=\int_0^1\frac{t^{p+k}\dd t}{(1+t)^{2+k}}. 
\end{equation}
Here and subsequently, it is assumed that $p\in(0,\infty)$ and $k\in\{0,1,\dots\}$, unless otherwise indicated. 
Further integrating by parts, one %observes that %
obtains the recurrence relation  
\begin{equation*}
J_{p,k}=\frac{2^{-2 - k} + (2 + k)J_{p,k+1}}{p + k + 1}, % wallis2.nb
\end{equation*}
whence, by induction, 
\begin{equation*}%\label{eq:J_p0=}
	J_{p,0}=p\,\ell_k(p) +p\,R_{p,k},  
\end{equation*}
where $\ell_k(p)$ is as in \eqref{eq:ell,u} and 
\begin{align}
R_{p,k}&:=\frac{(k+1)!}{p^{(k+1)}} J_{p,k}. %\quad\text{and} 
\label{eq:R_pk=} %\\
\end{align}
%Toward that end, let us bound $J_{p,k}$. 
So, by \eqref{eq:lnr'=J}, one has 
\begin{equation}\label{eq:lnr'=ell+R}
	\fd{\ln r(p)}p-\ell_k(p)=R_{p,k}. 
\end{equation}

%Consider the first two 
For any $t\in(0,1)$ and $k\ge0$, the integrand 
$\frac{t^{p+k}}{(1+t)^{2+k}}$  
in \eqref{eq:J} does not exceed $\big(\frac t{1+t}\big)^k$, which is less than $2^{-k}$. 
Hence, $0<R_{p,k}<\frac{k+1}p\,2^{-k}\underset{k\to\infty}\longrightarrow0$ and hence, by \eqref{eq:lnr'=ell+R}, $\ell_k(p)\underset{k\to\infty}\longrightarrow\fd{\ln r(p)}p$, so that one has the first equality in \eqref{eq:lnr'-identity&ineqs}. 
On the other hand, \eqref{eq:r(p)r(p+1)}
implies 
\begin{equation}\label{eq:lnr'(p),lnr'(p+1)}
	\frac1{2p(p+1)}-\fd{\ln r(p+1)}p=\fd{\ln r(p)}p. 
\end{equation}
Therefore and because $\ell_k(p+1)\to\fd{\ln r(p+1)}p$, the definition of $u_k(p)$ in \eqref{eq:ell,u} yields $u_k(p)\to\fd{\ln r(p)}p$, as $k\to\infty$. 
In turn, this implies the second equality in \eqref{eq:lnr'-identity&ineqs}.  
The first inequality in \eqref{eq:psi-error} follows immediately from \eqref{eq:lnr'-identity&ineqs}. 

Next, observe that, by \eqref{eq:J}, $J_{p,k}$ is decreasing in $p$. 
Since the denominator of the fraction in \eqref{eq:R_pk=} is positive and increasing in $p>0$, one sees that the ``remainder'' $R_{p,k}$ is decreasing in $p>0$. 
So, in view of    
the definition of $u_k(p)$ in 
\eqref{eq:ell,u} and the identities \eqref{eq:lnr'(p),lnr'(p+1)} and \eqref{eq:lnr'=ell+R}, 
\begin{equation*}
	u_k(p)-\fd{\ln r(p)}p=\fd{\ln r(p+1)}p-\ell_k(p+1)=R_{p+1,k}<R_{p,k}=\fd{\ln r(p)}p-\ell_k(p), 
\end{equation*}
whence the second inequality in \eqref{eq:psi-error} follows. 

The last inequality in \eqref{eq:psi-error} is obvious. 
%To complete the proof of Proposition~\ref{prop:r'}, it remains to prove the third, penultimate inequality in \eqref{eq:psi-error}. 
%Toward that end, 
Finally, introduce $a_m:=\frac{2^{-1-m} m!}{p^{(m+1)}}$ and 
note that $\frac{a_{m+1}}{a_m}<\frac12$ for all natural $m$, so that 
\begin{equation*}
	\fd{\ln r(p)}p-\ell_k(p)
	=\ell_\infty(p)-\ell_k(p)
	=\sum_{m=k+1}^\infty a_m<2a_{k+1},  
\end{equation*}
which proves the third, penultimate inequality in \eqref{eq:psi-error} and thus completes the proof of Proposition~\ref{prop:r'}. 
\end{proof}

\begin{proof}[Proof of Theorem~\ref{th:}] 
Take any $p\in(0,\infty)$ and $k\in\{0,1,\dots\}$. 
Observe that $\frac{m!}{p^{(m+1)}}=
\sum _{j=0}^m \frac{(-1)^j \binom{m}{j}}{p+j}$ for all natural $m$; 
this can be checked by induction in $m$ or as follows: 
\begin{align*}
	\sum _{j=0}^m \frac{(-1)^j \binom{m}{j}}{p+j}
	&=\sum _{j=0}^m (-1)^j \binom{m}{j}\int_0^\infty e^{-(p+j)u}\dd u \\
	&=\int_0^\infty e^{-pu} \sum _{j=0}^m \binom{m}{j} (-e^{-u})^j \dd u \\
	&=\int_0^\infty e^{-pu} (1-e^{-u})^m \dd u \\
	&=\int_0^1 z^{p-1} (1-z)^m \dd z
	=\frac{\Ga(p)\Ga(m+1)}{\Ga(p+m+1)}
	=\frac{m!}{p^{(m+1)}}. 
\end{align*}
So, by \eqref{eq:ell,u}, 
\begin{equation*}
	\ell_k(p)=\sum _{m=1}^k 2^{-1-m}\,\sum _{j=0}^m \frac{(-1)^j \binom{m}{j}}{p+j}. 
\end{equation*}
Recalling now \eqref{eq:U} and \eqref{eq:si}, 
one has 
\begin{equation}\label{eq:int ell}
	\int_p^\infty\ell_k(s)\dd s=-\ln U_k(p);  
\end{equation}
here, it was taken into account that 
$\sum _{j=0}^m (-1)^j \binom{m}{j}=0$ and hence \break 
$\sum _{j=0}^m (-1)^j \binom{m}{j}\,%\break
\ln(s+j)
=\sum _{j=0}^m (-1)^j \binom{m}{j}\,[\ln(s+j)-\ln s]\underset{s\to\infty}\longrightarrow0$, for each $m=1,2,\dots$. 
In view of \eqref{eq:ell,u} and \eqref{eq:LL}, one similarly has 
\begin{equation}\label{eq:int u}
\begin{aligned}
		\int_p^\infty u_k(s)\dd s
		=&\int_p^\infty\frac{\dd s}{2s(s+1)}-\int_p^\infty\ell_k(s+1)\dd s \\
		=&-\ln\sqrt{\frac p{p+1}}+\ln U_k(p+1)
		=-\ln L_k(p).
\end{aligned} 	
\end{equation}
Also, it is easy to see that $r(p)\to1$ as $p\to\infty$. 
Therefore, % and in view of \eqref{eq:lnr'=ell+R}, 
\begin{align}\label{eq:ln r(p)=}
	\ln r(p)=-\int_p^\infty\fd{\ln r(s)}s \dd s. 
\end{align} 
So, the 
inequalities \eqref{eq:L<r<U} follow by integrating those in \eqref{eq:lnr'-identity&ineqs} in view of \eqref{eq:int ell}, \eqref{eq:int u}, and \eqref{eq:ln r(p)=}. 
Similarly integrating inequalities \eqref{eq:psi-error},  
one sees that 
\begin{align*}
	0<\ln\frac{r(p)}{L_k(p)}<\ln\frac{U_k(p)}{r(p)}
	<&\frac{(k+1)!}{2^{k+1}}\,\int_p^\infty\frac{\dd s}{s^{(k+2)}} \\
	<&\frac{2^{-k-1}(k+1)!}{(p+2)^{(k)}}\,\int_p^\infty\frac{\dd s}{s(s+1)} \\ 
	=&\frac{2^{-k-1}(k+1)!}{(p+2)^{(k)}}\,\ln\frac{p+1}p
	<\rho^*_{p,k}, 
\end{align*}
where $\rho^*_{p,k}$ is as in \eqref{eq:rrho}. 
This 
%\big(together with the inequality $1-e^{-\rho}\le e^\rho-1$ for $\rho\in\R$\big)
yields \eqref{eq:err bound} and also implies that $L_k(p)$ and $U_k(p)$ both converge to $r(p)$ as $k\to\infty$, so that the identities \eqref{eq:identity} follow as well. 
Thus, Theorem~\ref{th:} is proved. 
\end{proof}

\bibliographystyle{abbrv}
%\bibliographystyle{ims}
%\bibliography{are.citations}
%\bibliography{citat}

%\bibliography{citations}

\bibliography{C:/Users/Iosif/Documents/mtu_home01-30-10/bib_files/citations}
%\bibliography{C:/Users/Iosif/Documents/mtu_home12-22-08/bib_files/citations}

\end{document}